\documentclass[12pt,a4paper]{amsart}
\usepackage{amssymb,amsmath}

\theoremstyle{plain} 
\newtheorem{theorem}{Theorem}[section]
\newtheorem{lemma}[theorem]{Lemma}

\newtheorem{proposition}[theorem]{Proposition}

\numberwithin{equation}{section}

\newcommand{\N}{\mathbb{N}}
\newcommand{\Z}{\mathbb{Z}}
\newcommand{\Q}{\mathbb{Q}}

\newcommand{\coloneqq}{\mathrel{\mathop:}=}

\DeclareMathOperator{\Aut}{Aut} 
\DeclareMathOperator{\Span}{span}
\DeclareMathOperator{\img}{img}
\DeclareMathOperator{\Mat}{Mat}
\DeclareMathOperator{\GL}{GL}
\DeclareMathOperator{\End}{End}

\renewcommand{\rho}{\varrho} 
\renewcommand{\theta}{\vartheta}

\begin{document}

\title[A nilpotent group without functional equations]{A nilpotent
  group without local functional equations for pro-isomorphic
  subgroups}

\author{Mark N.\ Berman}\address{Department of Mathematics, Ort Braude
  College, P.O. Box 78, Snunit St., 51, Karmiel 2161002, Israel}
\email{mark.n.berman@gmail.com}

\author{Benjamin Klopsch} \address{Mathematisches Institut,
  Heinrich-Heine-Universit\"at D\"usseldorf, 40225 D\"usseldorf, Germany}
\email{klopsch@math.uni-duesseldorf.de}



\begin{abstract}
  The pro-isomorphic zeta function $\zeta_\Gamma^\wedge(s)$ of a
  torsion-free finitely generated nilpotent group $\Gamma$ enumerates
  finite index subgroups $\Delta \leq \Gamma$ such that $\Delta$ and
  $\Gamma$ have isomorphic profinite completions.  It admits an Euler
  product decomposition $\zeta_\Gamma^\wedge(s) = \prod_p
  \zeta_{\Gamma,p}^\wedge(s)$.

  We manufacture the first example of a torsion-free finitely
  generated nilpotent group $\Gamma$ such that the local Euler factors
  $\zeta_{\Gamma,p}^\wedge(s)$ do not satisfy functional equations.
  The group $\Gamma$ has nilpotency class $4$ and Hirsch length~$25$.
  It is obtained, via the Malcev correspondence, from a $\Z$-Lie
  lattice $\Lambda$ with a suitable algebraic automorphism group
  $\mathbf{Aut}(\Lambda)$.
\end{abstract}

\maketitle


\section{Introduction}

Let $\Gamma$ be a torsion-free finitely generated nilpotent group.  In
analogy to classical zeta functions such as the Dedekind zeta function
of a number field, Grunewald, Segal and Smith introduced
in~\cite{GrSeSm88} zeta functions counting certain finite index
subgroups of $\Gamma$.  More precisely, they defined Dirichlet
generating functions
\begin{align*}
  \zeta_\Gamma^\leq(s) & = \sum_{n=1}^\infty
  \frac{a_n^\leq(\Gamma)}{n^s}, && \zeta_\Gamma^\trianglelefteq(s) =
  \sum_{n=1}^\infty
  \frac{a_n^\trianglelefteq(\Gamma)}{n^s}, \\
  \zeta_\Gamma^{\mathrm{iso}}(s) & = \sum_{n=1}^\infty
  \frac{a_n^{\mathrm{iso}}(\Gamma)}{n^s}, && \zeta_\Gamma^\wedge(s) =
  \sum_{n=1}^\infty \frac{a_n^\wedge(\Gamma)}{n^s},
\end{align*}
where $a_n^\leq(\Gamma)$, $a_n^\trianglelefteq(\Gamma)$,
$a_n^\mathrm{iso}(\Gamma)$ and $a_n^\wedge(\Gamma)$ denote the number
of subgroups $\Delta$ of index $n$ in $\Gamma$ satisfying $\Delta \leq
\Gamma$, $\Delta \trianglelefteq \Gamma$, $\Delta \cong \Gamma$ and
$\widehat{\Delta} \cong \widehat{\Gamma}$ respectively.  Here
$\widehat{H}$ denotes the profinite completion of a group~$H$.  The
Dirichlet series above are now commonly referred to as the
\emph{subgroup zeta function}, the \emph{normal zeta function}, the
\emph{isomorphic zeta function} and the \emph{pro-isomorphic zeta
  function}.

The group $\Gamma$ being nilpotent, one easily derives Euler product
decompositions
\begin{equation*} \zeta_\Gamma^\leq(s) =
  \prod_p \zeta_{\Gamma,p}^\leq(s), \quad
  \zeta_\Gamma^\trianglelefteq(s) = \prod_p
  \zeta_{\Gamma,p}^\trianglelefteq(s), \quad \zeta_\Gamma^\wedge(s) =
  \prod_p \zeta_{\Gamma,p}^\wedge(s),
\end{equation*}
where each product extends over all rational primes $p$ and
\begin{equation} \label{equ:local-factor} \zeta_{\Gamma,p}^*(s) =
  \sum_{k=0}^\infty a_{p^k}^*(\Gamma) p^{-ks}, \qquad \text{for $*$
    one of $\leq$, $\trianglelefteq$, $\wedge$},
\end{equation}
is the \emph{local zeta function} at a prime~$p$.  In contrast, the
zeta function $\zeta_\Gamma^{\mathrm{iso}}(s)$ generally does not
admit a decomposition of this kind.  One of the key results
in~\cite{GrSeSm88} is that each of the local factors defined
in~\eqref{equ:local-factor} is in fact a rational function over $\Q$
in $p^{-s}$.

Over the last $20$ years many important advances have been made in the
study of zeta functions of nilpotent groups; for instance, see
\cite{dSGr00,dSWo08,Vo10}.  One of the prominent questions driving the
subject has been whether the local zeta functions satisfy functional
equations: is it true that
\[
\zeta_{\Gamma,p}^*(s) \vert_{p \rightarrow p^{-1}} = (-1)^a p^{b-cs}
\zeta_{\Gamma,p}(s) \qquad \text{for almost all $p$},
\]
where $a,b,c$ are certain integer parameters depending on $\Gamma$?
In~\cite{Vo10}, Voll derived positive answers for subgroup zeta
functions in general and for normal zeta functions associated to
groups of nilpotency class at most~$2$.  More precisely, he showed
that the local zeta functions in question can be expressed as rational
functions in $p^{-s}$ whose coefficients involve the numbers $b_V(p)$
of $\mathbb{F}_p$-rational points of certain smooth projective
varieties~$V$.  The number $b_V(p^{-1})$ is obtained by writing
$b_V(p)$ as an alternating sum of Frobenius eigenvalues and
subsequently inverting these eigenvalues.  It is known that normal
zeta functions of nilpotent groups of class $3$ may or may not satisfy
local functional equations as above; see~\cite[Section~2.11]{dSWo08}.

In comparison to $\zeta_{\Gamma,p}^\leq(s)$ and
$\zeta_{\Gamma,p}^\trianglelefteq(s)$, our present picture of the
local pro-iso\-mor\-phic zeta function $\zeta_{\Gamma,p}^\wedge(s)$ is
somewhat less complete.  This is perhaps due to the fact that
$\zeta_{\Gamma,p}^\wedge(s)$ depends crucially on an intermediate
object that is generally not easy to pin down, namely the algebraic
automorphism group $\mathbf{Aut}(\Lambda)$ of a $\Z$-Lie lattice
$\Lambda$ naturally associated to~$\Gamma$.  Using
$\mathbf{Aut}(\Lambda)$, one can compute $\zeta_{\Gamma,p}^\wedge(s)$
as a $p$-adic integral similar to zeta functions of reductive
algebraic groups that were studied in the 1960s by Weil, Tamagawa,
Satake and Macdonald; cf.\ Proposition~\ref{pro:integral-formula}.
Building upon work of Igusa~\cite{Ig89}, du Sautoy and
Lubotzky~\cite{dSLu96} and Berman~\cite{Be11} have established
functional equations for $\zeta_{\Gamma,p}^\wedge(s)$ subject to
certain conditions on $\mathbf{Aut}(\Lambda)$.
 
The purpose of this paper is to give the first example of a nilpotent
group $\Gamma$ such that the local pro-isomorphic zeta functions
$\zeta_{\Gamma,p}(s)$ do not satisfy functional equations in the sense
discussed above.

\begin{theorem} \label{thm:main-thm} There exists a torsion-free
  finitely generated nilpotent group $\Gamma$, of nilpotency class~$4$
  and Hirsch length $25$, such that, for all primes $p>3$,
  \[
  \zeta_{\Gamma,p}^\wedge(s) = \frac{1 + p^{285-102s} + 2p^{286-102s}
    + 2p^{572-204s}}{(1 - p^{285-102s})(1 - p^{573-204s})}.
  \]
\end{theorem}

In particular, this resolves Question~1.3 in~\cite{Be11}.  As
explained in Section~\ref{sec:Malcev}, the proof of
Theorem~\ref{thm:main-thm} reduces to the construction of a suitable
nilpotent $\Z$-Lie lattice~$\Lambda$, linked to a group $\Gamma$ via
the Malcev correspondence.  In Section~\ref{sec:descr.of.L} we
describe a candidate for such a Lie lattice~$\Lambda$, in terms of
generators and relations: $\Lambda$ has nilpotency class~$4$ and
$\Z$-rank~$25$; its construction is motivated by certain integrals
presented in \cite[Section~6]{Be11}.  In Section~\ref{sec:aut-comp} we
carry out the task of pinning down the algebraic automorphism group
$\mathbf{Aut}(\Lambda)$.  In Section~\ref{sec:zeta-comp} we compute
$\zeta_{\Gamma,p}^\wedge(s)$, using our description of
$\mathbf{Aut}(\Lambda)$ and the machinery developed in~\cite{dSLu96}.


\section{Reduction to $\Z$-Lie lattices} \label{sec:Malcev}

Let $\Gamma$ be a torsion-free finitely generated nilpotent group.
Then the profinite completion $\widehat{\Gamma} \cong \prod_p
\widehat{\Gamma}_p$ is pro-nilpotent, with Sylow pro-$p$ subgroups
isomorphic to the pro-$p$ completions $\widehat{\Gamma}_p$, and
\[
\zeta_\Gamma^\wedge(s) = \zeta_{\widehat{\Gamma}}^{\mathrm{iso}}(s) =
\prod_{p} \zeta_{\widehat{\Gamma}_p}^{\mathrm{iso}}(s).
\]
One of the key steps in~\cite{GrSeSm88} is to use the Malcev
correspondence to `linearise' the problem of computing the factors
$\zeta_{\widehat{\Gamma}_p}^{\mathrm{iso}}(s)$ by passing from groups
to Lie lattices.

Let $\Lambda$ be a $\Z$-Lie lattice.  In analogy to the zeta functions
defined for groups, the \emph{isomorphic zeta function} and the
\emph{pro-isomorphic zeta function} of $\Lambda$ are
\[
\zeta_\Lambda^{\mathrm{iso}} (s) = \sum_{n=1}^\infty
\frac{a_n^{\mathrm{iso}}(\Lambda)}{n^s}, \qquad \zeta_\Lambda^\wedge
(s) = \sum_{n=1}^\infty \frac{a_n^\wedge(\Lambda)}{n^s},
\]
where $a_n^{\mathrm{iso}}(\Lambda)$ and $a_n^\wedge(\Lambda)$ denote
the number of Lie sublattices $M$ of index $n$ in $\Lambda$ satisfying
$M \cong \Lambda$ and $\widehat{\Z} \otimes_\Z M \cong \widehat{\Z}
\otimes_\Z \Lambda$ respectively.  As $\widehat{\Z} \cong \prod_p
\Z_p$, where $\Z_p$ denotes the ring of $p$-adic integers, the latter
condition is equivalent to: $\Z_p \otimes_\Z M \cong \Z_p \otimes_\Z
\Lambda$ for all primes~$p$.  Moreover,
\[
\zeta_\Lambda^\wedge(s) = \zeta_{\widehat{\Z} \otimes_\Z
  \Lambda}^{\mathrm{iso}}(s) = \prod_{p} \zeta_{\Z_p \otimes_\Z
  \Lambda}^{\mathrm{iso}}(s).
\]

In~\cite[Section~4]{GrSeSm88} one finds a discussion of how starting
from a group $\Gamma$ one obtains a $\Z$-Lie lattice $\Lambda$ such
that $\zeta_{\widehat{\Gamma}_p}^{\mathrm{iso}}(s) = \zeta_{\Z_p
  \otimes_\Z \Lambda}^{\mathrm{iso}}(s)$ for almost all primes~$p$.
Our aim in this section is to complement the treatment
in~\cite{GrSeSm88}, by giving a detailed account of the transition in
the opposite direction from $\Z$-Lie lattices to groups; refer
to~\cite[Chapter~6]{Se83} for an alternative approach
and~\cite[Chapter~10]{Kh98} for a description of the related Malcev
correspondence.  We make use of the Lie correspondence between
$p$-adic analytic pro-$p$ groups and $\Z_p$-Lie lattices, which --
similarly to the Malcev correspondence -- is effected by the Hausdorff
series
\[
\Phi(X,Y) = \log (\exp(X) \exp(Y)) \in \Q \langle\!\langle X,Y
\rangle\!\rangle,
\]
a formal power series in non-commuting variables $X,Y$, where
\[
\exp(Z) = \sum_{n=0}^\infty \frac{Z^n}{n!} \qquad \text{and} \qquad
\log(Z) = \sum_{n=1}^\infty (-1)^{n+1} \frac{(Z-1)^n}{n}.
\]
By arranging terms suitably, one can write $\Phi(X,Y)$ as a sum
\begin{equation} \label{equ:Hausdorff-series} \Phi(X,Y) =
  \sum_{n=1}^\infty u_n(X,Y)
\end{equation}
of homogeneous Lie polynomials $u_n(X,Y)$ in $X,Y$ of total degree $n$
with rational coefficients.  For instance, modulo terms of total
degree at least $5$, the Hausdorff series is congruent to the Lie
polynomial
\[
\Phi_4(X,Y) = X + Y + \tfrac{1}{2} [X,Y] - \tfrac{1}{12} [X,Y,X] +
\tfrac{1}{12} [X,Y,Y] - \tfrac{1}{24} [X,Y,X,Y].
\]
For $c \in \N$, let $m(c)$ denote the least common denominator of the
rational coefficients appearing in \eqref{equ:Hausdorff-series} up to
total degree~$c$; for instance, $m(4) = 24$.

Suppose that $\Lambda$ is a finitely generated nilpotent $\Z$-Lie
lattice of class $c$, and let $\Phi_c(X,Y)$ denote the Lie polynomial
obtained from $\Phi(X,Y)$ by truncating after terms of total degree at
most~$c$.  Setting $m = m(c)$, we have $[m\Lambda,m\Lambda] \subseteq
m (m\Lambda)$.  Consequently,
\[
\mathsf{exp}(m\Lambda) \coloneqq (m\Lambda,*), \qquad \text{where $x*y
  = \Phi_c(x,y)$,}
\]
defines a torsion-free finitely generated nilpotent group of
class~$c$.  Indeed, the formal identities
\begin{align*}
  \Phi(\Phi(X,Y),Z) & = \Phi(X,\Phi(Y,Z)), \\ \Phi(0,X) = \Phi(X,0) = X,
         \quad & \Phi(-X,X) = \Phi(X,-X) = 0
\end{align*}
show that $\mathsf{exp}(m\Lambda)$ is a group.  Furthermore, the
identity $\Phi(X,X') = X + X'$ for commuting variables $X,X'$ shows
that on every abelian Lie sublattice $A$ of $m\Lambda$ the operation
$*$ is the same as Lie addition; in particular,
$\mathsf{exp}(m\Lambda)$ is torsion-free, and central isolated Lie
sublattices correspond to central isolated subgroups.  By induction on
the nilpotency class~$c$, one shows that $\mathsf{exp}(m\Lambda)$ is
finitely generated nilpotent of class at most $c$ and has Hirsch
length equal to the $\Z$-rank of $m\Lambda$.  The nilpotency class of
$\mathsf{exp}(m\Lambda)$ is equal to $c$, because each group
commutator $[x_1, \ldots, x_c]_\mathrm{grp}$ of length $c$ yields
exactly the same element as the corresponding Lie commutator $[x_1,
\ldots, x_c]_\mathrm{Lie}$; this follows, by induction, from the
congruence $\Phi(\Phi(-X,-Y),\Phi(X,Y)) \equiv [X,Y]$ modulo
commutators of length at least~$3$.  This establishes the first half
of the following proposition.
 
\begin{proposition} \label{pro:Malcev-corr}
  Let $\Lambda$ be a nilpotent $\Z$-Lie lattice of class $c$, and set
  $m = m(c)$.  Then $\Gamma = \mathsf{exp}(m \Lambda)$ is a
  torsion-free finitely generated nilpotent group of class~$c$, and
  for all primes $p$ with $p \nmid m$,
  \[
  \zeta_{\widehat{\Gamma}_p}^{\mathrm{iso}}(s) = \zeta_{\Z_p
    \otimes_\Z \Lambda}^{\mathrm{iso}}(s).
  \]
\end{proposition}

\begin{proof}
  It remains to show that the equality between the two zeta functions
  holds for all primes $p$ with $p \nmid m$.  Let $p$ be such a prime.
  We observe that $\Phi_c(X,Y)$ has coefficients in~$\Z_p$.  Write $L
  = \Z_p \otimes_\Z \Lambda = \Z_p \otimes_\Z m \Lambda$ and $G =
  \mathsf{exp}(L) = (L,*)$, where $x*y = \Phi_c(x,y)$ as above.
  Observe that $L$ and $G$ have the same underlying sets and that $G$,
  with the topology inherited from $L$, is a topological group.  To
  distinguish Lie and group commutators we write $[x,y]_\mathrm{Lie}$
  and $[x,y]_\mathrm{grp}$ respectively.

  Observe that $G^{p^k} = \langle x^{p^k} \mid x \in G \rangle$ is
  equal to $G^{\{p^k\}} = \{ x^{p^k} \mid x \in G \} =
  \mathsf{exp}(p^k L)$ for each $k \in \N$, because $p^k L$ is closed
  under~$*$.  Since the sublattices $p^k L$ form a base for the
  neighbourhoods of $0$, the group $G$ is a pro-$p$ group.  Similarly,
  one sees that $\Gamma^{p^k} = \Gamma^{\{p^k\}} =
  \mathsf{exp}(p^km\Lambda)$.  This implies
  \[
  \widehat{\Gamma}_p = \varprojlim \Gamma / \Gamma^{p^k} = \varprojlim
  \mathsf{exp}(m\Lambda) / \mathsf{exp}(p^k m\Lambda) \cong
  \varprojlim \mathsf{exp}(L) / \mathsf{exp}(p^k L) \cong
  \mathsf{exp}(L) = G.
  \]

  To finish the proof we verify that the construction $M \mapsto
  \mathsf{exp}(M)$ sets up an index-preserving one-to-one
  correspondence
  \[
  \{ M \mid M \leq L \text{ a Lie sublattice, $M \cong L$} \}
  \rightarrow \{ H \mid H \leq G \text{ a subgroup, $H \cong G$} \}.
  \]
  Since $\Phi_c(X,Y)$ has coefficients in $\Z_p$, every $\Z_p$-Lie
  sublattice $M$ of $L$ gives rise to a subgroup $\mathsf{exp}(M)$.
  Moreover, the construction is functorial so that $M \cong L$ implies
  $\mathsf{exp}(M) \cong \mathsf{exp}(L) = G$.  Conversely, suppose
  that $H$ is a subgroup of $G$ with $H \cong G$.  The group $G$ has
  the properties
  \[
  \text{(i) } G^{\{p^k\}} = G^{p^k} \text{ and } [G^{p^k},G^{p^k}]
  \subseteq G^{p^{2k}} \text{ for } k \in \N, \quad \text{(ii) } G
  \rightarrow G, x \mapsto x^p \text{ is injective.}
  \]
  Based on these, one can recover the Lie operations on $L$ from the
  group structure of $G$ by means of
  \begin{equation} \label{equ:limit-desc} x + y = \lim_{k \to \infty}
    (x^{p^k} * y^{p^k})^{p^{-k}} \qquad \text{and} \qquad
    [x,y]_{\mathrm{Lie}} = \lim_{k \to \infty}
    [x^{p^k},y^{p^k}]_{\mathrm{grp}}^{\, p^{-2k}}.
  \end{equation}
  Since $H \cong G$, the relations~\eqref{equ:limit-desc} applied to
  $H$ show that the set $H$ is closed under Lie addition and the Lie
  bracket.  Thus $H = \mathsf{exp}(M)$ for a Lie sublattice $M$ of
  $L$, and again by functoriality $M \cong L$.

  Finally, we claim that, for $x \in L$ and $k \in \N$, the Lie coset
  $x + p^k L$ and the group coset $x * G^{p^k}$ are equal as sets.  This
  will imply that the normalised Haar measures $\mu_{\mathrm{Lie}}$ on
  $L$ and $\mu_{\mathrm{grp}}$ on $G$ are the same so that for every
  Lie sublattice $M$ of $L$,
  \[
  \lvert L : M \rvert = \mu_{\mathrm{Lie}}(M)^{-1} =
  \mu_{\mathrm{grp}}(\mathsf{exp}(M))^{-1} = \lvert G :
  \mathsf{exp}(M) \rvert.
  \]
  To prove the claim, note that $x * G^{p^k} = \Phi_c(x,p^kL)
  \subseteq x + p^kL$.  The reverse inclusion follows inductively: one
  shows that $x + p^k \gamma_{c+1-i}(L) \subseteq x * G^{p^k}$ for $i
  \in \{1,\ldots,c\}$.
\end{proof}

Next we recall the notion of the \emph{algebraic automorphism group}
$\mathbf{Aut}(\Lambda)$ of a $\Z$-Lie lattice~$\Lambda$.  The group
$\mathbf{Aut}(\Lambda)$ is realised via a $\Z$-basis of $\Lambda$ as a
$\Q$-algebraic subgroup $\mathbf{G}$ of $\GL_d$, where $d$ denotes the
$\Z$-rank of $\Lambda$, so that
\[
\mathbf{G}(k) = \Aut_k(k \otimes_\Z L) \leq \GL_{d}(k)
\]
for every extension field $k$ of $\Q$.  It admits a natural arithmetic
structure so that
\[
\mathbf{G}(\Z) = \Aut(\Lambda) \quad \text{and} \quad \mathbf{G}(\Z_p)
= \Aut(\Z_p \otimes_\Z \Lambda) \text{ for every prime $p$.}
\]
We state~\cite[Proposition~3.4]{GrSeSm88}.

\begin{proposition}[Grunewald, Segal,
  Smith] \label{pro:integral-formula} Let $\Lambda$ be a nilpotent
  $\Z$-Lie lattice, with $\Z$-basis $\mathcal{E} = (e_1,\ldots,e_d)$
  say, and let $\mathbf{G} = \mathbf{Aut}(\Lambda) \subseteq \GL_d$
  denote the algebraic automorphism group of $\Lambda$, realised with
  respect to $\mathcal{E}$.  For each prime $p$, let
  \[
  G^+_p = \mathbf{G}(\Q_p) \cap \Mat_d(\Z_p) = \Aut(\Q_p \otimes_\Z
  \Lambda) \cap \End(\Z_p \otimes_\Z \Lambda)
  \]
  and let $\mu_p$ denote the right Haar measure on the locally compact
  group $\mathbf{G}(\Q_p)$ such that $\mu_p(\mathbf{G}(\Z_p)) =1$.
  Then for all primes $p$,
  \[
  \zeta_{\Z_p \otimes_\Z \Lambda}^{\mathrm{iso}}(s) = \int_{G^+_p}
  \lvert \det g \rvert_p^s \, d\mu_p(g).
  \]
\end{proposition}

In~\cite[Section~2]{dSLu96} we find a treatment of $p$-adic integrals
as in Proposition~\ref{pro:integral-formula}, subject to a series of
simplifying assumptions.  Almost all of those assumptions, relevant
for our purposes, are not restrictive, in the sense that they can be
realised by an appropriate equivalent representation of the
automorphism group corresponding to a different choice of
basis~$\mathcal{E}$; this is shown in~\cite[Section 4]{dSLu96}.  The
integral of Proposition~\ref{pro:integral-formula} is unaffected by
this change of representation for almost all primes.  It will turn out
that the assumptions required for our application in
Section~\ref{sec:zeta-comp} are automatically satisfied without the
need for an equivalent representation.  

We now outline the assumptions. Let $\mathbf{G} \subseteq \GL_d$ be an
affine group scheme over~$\Z$.  Decompose the connected component of
the identity as a semidirect product $\mathbf{G}^\circ = \mathbf{N}
\rtimes \mathbf{H}$ of the unipotent radical $\mathbf{N}$ and a
reductive group~$\mathbf{H}$.  Fix a prime $p$.  The first assumption
is that
\[
\mathbf{G}(\Q_p) = \mathbf{G}(\Z_p) \mathbf{G}^\circ(\Q_p);
\]
loosely speaking, this allows us to work with the connected group
$\mathbf{G}^\circ$ rather than~$\mathbf{G}$.  We
write $G = \mathbf{G}^\circ(\Q_p)$, $N = \mathbf{N}(\Q_p)$, $H =
\mathbf{H}(\Q_p)$.  Assume further that $G \subseteq \GL_d(\Q_p)$ is
in block form, where $H$ is block diagonal and $N$ is block upper
unitriangular in the following sense.  There is a partition $d = r_1 +
\ldots + r_c$ such that, setting $s_i = r_1 + \ldots + r_{i-1}$ for $i
\in \{1,\ldots,c\}$,
\begin{enumerate}
\item[$\circ$] the vector space $V = \Q_p^d$ on which $G$ acts from
  the right decomposes into a direct sum of $H$-stable subspaces $U_i
  = \{ (0,\ldots,0) \} \times \Q_p^{r_i} \times \{ (0,\ldots,0) \}$,
  where the vectors $(0,\ldots,0)$ have $s_i$, respectively
  $d-s_{i+1}$, entries;
\item[$\circ$] setting $V_i = U_i \oplus \ldots \oplus U_c$, each
  $V_i$ is $N$-stable and $N$ acts trivially on $V_i/V_{i+1}$.
\end{enumerate}
For each $i \in \{2,\ldots,c+1\}$ let $N_{i-1} = N \cap
\ker(\psi'_i)$, where $\psi'_i \colon G \rightarrow \Aut(V/V_i)$
denotes the natural action.  Let $\psi_i \colon G/N_{i-1} \rightarrow
\Aut(V/V_i)$ denote the induced map, and define
\[
(G/N_{i-1})^+ = \psi_i^{-1} \big(\psi_i(G/N_{i-1}) \cap
\Mat_{s_i}(\Z_p) \big).
\]
Assume that for every $g_0 N_{i-1} \in (G/N_{i-1})^+$ there exists $g
\in G^+$ such that $g_0 N_{i-1} = g N_{i-1}$; this is the crucial
`lifting condition'~\cite[Assumption~2.3]{dSLu96}.  As explained in
\cite[p.~6]{Be11}, this lifting condition cannot in general be
satisfied by moving to an equivalent representation. For $i \in
\{2,\ldots,c\}$ there is a natural embedding of $N_{i-1}/N_i
\hookrightarrow (V_i/V_{i+1})^{s_i}$ via the action of $N_{i-1}/N_i$
on $V/V_{i+1}$, recorded on the natural basis.  The action of $H$ on
$V_i/V_{i+1}$ induces, for each $h \in H$, a map
\[
\tau_i(h) \colon N_{i-1}/N_i \hookrightarrow (V_i/V_{i+1})^{s_i}.
\]
Define $\theta_{i-1} \colon H \rightarrow [0,1]$ by setting
\[
\theta_{i-1}(h) = \mu_{N_{i-1}/N_i} \big( \{n N_i \in N_{i-1}/N_i \mid (n
N_i) \tau_i(h) \in \Mat_{s_i,r_i}(\Z_p)\} \big),
\]  
where $\mu_{N_{i-1}/N_i}$ denotes the right Haar measure on
$N_{i-1}/N_i$, normalised such that the set $\psi_{i+1}^{-1}
(\psi_{i+1}(N_{i-1}/N_i) \cap \Mat_{s_{i+1}}(\Z_p))$ has measure~$1$.
Write $\mu_G$, respectively $\mu_H$, for the right Haar measure
on~$G$, respectively $H$, normalised such that
$\mu_G(\mathbf{G}(\Z_p)) = 1$ and $\mu_H( \mathbf{H}(\Z_p)) = 1$.
From $G = N \rtimes H$ one deduces that $\mu_G = \prod_{i=2}^c
\mu_{N_{i-1}/N_i} \cdot \mu_H$.  Setting $G^+ = G \cap \Mat_d(\Z_p)$
and $H^+ = H \cap \Mat_d(\Z_p)$, we can now
state~\cite[Theorem~2.2]{dSLu96}.

\begin{theorem}[du Sautoy and Lubotzky] \label{thm:dS-Lu} In the
  set-up described above, subject to the various assumptions,
  \[
  \int_{G^+} \lvert \det g \rvert_p^s \, d\mu_G (g) = \int_{H^+}
  \lvert \det h \rvert_p^s \, \prod_{i=1}^{c-1} \theta_i(h) \, d\mu_H(h).
  \]
\end{theorem}

For later use, we record a simple fact that is useful for detecting
when an element of~$G$, arising as in
Proposition~\ref{pro:integral-formula}, is integral.

\begin{lemma} \label{lem:enough-xyz} Suppose that $\Lambda$ is a
  $\Z$-Lie lattice of $\Z$-rank~$d$, generated by elements $x_1,
  \ldots, x_r$.  Put $L= \Z_p \otimes_\Z \Lambda$, and let $g \in G =
  \mathbf{G}(\Q_p)$, where $\mathbf{G} = \mathbf{Aut}(\Lambda)
  \subseteq \GL_d$.

  If $(x_1)g, \ldots, (x_r)g \in L$, then $g \in G^+ =
  \mathbf{G}(\Q_p) \cap \Mat_d(\Z_p)$.
\end{lemma}

\begin{proof}
  Note that $g\in G^+$ if and only if $(z)g\in L$ for all $z \in L$.
  Suppose that $(x_1)g, \ldots, (x_r)g \in L$.  Let $w \in L$ be an
  arbitrary commutator in $x_1,\ldots, x_r$, i.e.\ $w = [y_1, \ldots,
  y_n]$ with $y_i \in \{x_1,\ldots,x_r\}$ for $1 \leq i \leq n$.  Then
  $(w)g = [(y_1)g, \cdots, (y_n)g]$ is a commutator of elements of
  $L$, hence lies in~$L$.  By linearity, $(z)g \in L$ for all $z \in
  L$, since $x_1, \ldots, x_r$ generate $L$ as a $\Z_p$-Lie lattice.
  Thus $g\in G^+$.
\end{proof}


\section{The $\Z$-Lie lattice $\Lambda$} \label{sec:descr.of.L} Let
$\mathcal{F}$ be the free nilpotent $\Z$-Lie ring of class $4$ on
generators $X, Y, Z$.  In the following we will primarily be dealing
with Lie rings and it is convenient to use simple product notation for
the Lie bracket and left-normed notation.  Thus we write $XYZ$ in
place of $[[X,Y],Z]$.  The Hall collection process (see
Appendix~\ref{sec:appendix_A}) yields the following graded $\Z$-basis
for $\mathcal{F}$:
\begin{equation}\label{equ:F_basis}
  \begin{split}
    & X,\, Y,\, Z, \qquad
    XY,\, XZ,\, YZ, \\
    & XYY,\, XZZ,\, XYZ,\, XZY,\, XYX,\, XZX,\,
    \underline{YZY},\, \underline{ZYZ}, \\
    & XYYY,\, XZZZ,\, XYXX,\, XZXX,\, XYXY,\, XZXZ, \\
    & XYXZ,\, XZXY,\, XYZX,\, XYZZ,\, XZYY,\, XYZY,\\  
    & XZYZ,\,\underline{XYYZ}, \underline{XZZY},\, 
    \underline{YZYY},\, \underline{YZYZ},\, \underline{ZYZZ}.
  \end{split}
\end{equation}
In particular, $\dim_\Q (\Q \otimes_\Z \mathcal{F}) = 32$.  

Let $\mathcal{I}$ be the Lie ideal of $\mathcal{F}$ generated by the
two relations
\[
R_1 = YXXX - YZY \quad \text{and} \quad R_2 = ZXXX - ZYZ.
\]
We determine a basis for~$\mathcal{I}$.  Examining terms of the form
$R_1 S$ and $R_2 S$ for $S \in \{X, Y, Z\}$, we obtain a spanning set
for $\mathcal{I}$, namely
\[
R_1, R_2, \quad YZYX,\, YZYY,\, YZYZ, \quad ZYZX,\, ZYZY,\, ZYZZ.
\]
We make repeated use of the following basic Lie identity, cf.\
\eqref{equ:PQRS-appendix} in Appendix~\ref{sec:appendix_A}:
\begin{equation} \label{equ:PQRS} 
  PQRS = PSQR + SQPR + RSQP + SRPQ.
\end{equation}
The substitution $(P,Q,R,S) = (Y,Z,Y,Z)$ yields $YZYZ = 2YZZY +
ZYYZ$, hence
\[
YZYZ = YZZY = -ZYZY.
\]
This shows that we may omit $ZYZY$ from the spanning set of
$\mathcal{I}$.  We claim that the remaining seven elements 
\[
R_1, R_2, \quad YZYX,\, ZYZX,\quad \underline{YZYY},\,
\underline{YZYZ},\, \underline{ZYZZ}.
\]
form a $\Z$-basis for $\mathcal{I}$; in particular, $\dim_\Q (\Q
\otimes_\Z \mathcal{I}) = 7$.  Indeed, we only need to look at the
elements of weight~$4$.  The last three of them are part of the
basis~\eqref{equ:F_basis}, and evaluating~\eqref{equ:PQRS} for
$(P,Q,R,S)=(Y,Z,Y,X)$ and $(P,Q,R,S)=(Z,Y,Z,X)$ we can express the
remaining two spanning elements also in terms of the
basis~\eqref{equ:F_basis}:
\begin{equation} \label{equ:YZYX-identity}
  \begin{split}
    YZYX & = -2 XYZY + XZYY + \underline{XYYZ}, \\
    ZYZX & = -2 XZYZ + XYZZ + \underline{XZZY}.
  \end{split}
\end{equation}

Set $\Lambda = \mathcal{F}/\mathcal{I}$ and write $x, y, z$ for the
images of $X, Y, Z$ in~$\Lambda$.  From our description of
$\mathcal{I}$ we conclude that $\dim_\Q (\Q \otimes_\Z \Lambda) = 25$
and that $\Lambda$ admits the following $\Z$-basis, obtained by
omitting the images of the underlined elements in~\eqref{equ:F_basis}
matching the underlined terms in the expressions above:
 
\begin{equation} \label{equ:L_basis}
  \begin{split}
    & x,\, y,\, z, \quad
    xy,\, xz,\, yz, \\
    & xyy,\, xzz,\, xyz,\, xzy,\, xyx,\, xzx,\, \\
    & xyyy,\, xzzz,\, xyxx,\, xzxx,\, xyxy,\, xzxz, \\
    & xyxz,\, xzxy,\, xyzx,\, xyzz,\, xzyy,\, xyzy,\, xzyz.
\end{split}
\end{equation}
For later use we record two consequences of our discussion,
see~\eqref{equ:YZYX-identity}:
\begin{equation}
  \label{equ:xyzy-identity}
  2 xyzy = xzyy + xyyz \qquad \text{and} \qquad
  2 xzyz = xyzz + xzzy.
\end{equation}

Finally, we make the following
observation, which can be checked by direct calculation using standard
identities and~\eqref{equ:PQRS}; in verifying the claim, it is
convenient to consider first the analogous statement for $\mathcal{F}$
and subsequently pass to the quotient~$\Lambda$.

\begin{lemma} \label{lem:same-weight} Every commutator $w$ of length
  $4$ in $x,y,z \in \Lambda$ can be expressed uniquely as a
  $\Z$-linear combination of elements of the basis~\eqref{equ:L_basis}
  having the same weights for $x,y,z$; that is: $x,y,z$ appear in each
  of these basis elements with the same multiplicity as in the
  original~$w$.
\end{lemma}


\section{The algebraic automorphism group
  $\mathbf{Aut}(\Lambda)$} \label{sec:aut-comp}

Using the $\Z$-basis of $\Lambda$ listed in~\eqref{equ:L_basis}, the
algebraic automorphism group $\mathbf{Aut}(\Lambda)$ of $\Lambda$ can
be realised as a $\Q$-algebraic subgroup $\mathbf{G}$ of $\GL_{25}$ so
that
\[
\mathbf{G}(k) = \Aut_k(k \otimes_\Z \Lambda) \leq \GL_{25}(k)
\]
for every extension field $k$ of~$\Q$.  Moreover,
\[
\mathbf{G}(\Z) = \Aut(\Lambda) \quad \text{and} \quad \mathbf{G}(\Z_p)
= \Aut(\Z_p \otimes_\Z \Lambda) \text{ for every prime $p$.}
\]
We observe that for every extension field $k$ of $\Q$,
\begin{enumerate}
\item every $k$-automorphism of $k \otimes_\Z \Lambda$ lifts to a
  $k$-automorphism of $k \otimes_\Z \mathcal{F}$, and $\alpha \in
  \Aut_k(k \otimes_\Z \mathcal{F})$ induces an automorphism $\alpha_{k
    \otimes \Lambda} \in \Aut_k(k \otimes_\Z \Lambda)$ if and only if
  $(k \otimes_\Z \mathcal{I}) \alpha \subseteq k \otimes_\Z
  \mathcal{I}$;
\item every $\alpha \in \Aut_k(k \otimes_\Z \mathcal{F})$ is uniquely
  determined by the images $X\alpha$, $Y\alpha$, $Z\alpha$; conversely
  any choice $X_0$, $Y_0$, $Z_0 \in k \otimes_\Z \mathcal{F}$ with $k
  \otimes_\Z \mathcal{F} = \Span_k \langle X_0, Y_0, Z_0 \rangle +
  \gamma_2(k \otimes_\Z \mathcal{F})$ yields a
  $k$\nobreakdash-automorphism $\alpha$ of $k \otimes_\Z \mathcal{F}$
  such that $X\alpha = X_0$, $Y\alpha = Y_0$, $Z\alpha =Z_0$.
\end{enumerate}
Consequently, we identify two natural subgroups of
$\mathbf{Aut}(\Lambda)$.  The affine algebraic group
\[
\mathbf{T} = \left\{
  \left( \begin{smallmatrix} 
    \lambda & 0 & 0 \\
    0 & \mu & 0 \\
    0 & 0 & \nu
  \end{smallmatrix} \right)
  \in \GL_{3} \mid \lambda^3 = \mu \nu \right\} \rtimes \left\{
  \left( \begin{smallmatrix} 
    1&0&0 \\
    0&1&0 \\
    0&0&1
  \end{smallmatrix} \right),
  \left( \begin{smallmatrix}
    1&0&0\\
    0&0&1\\
    0&1&0
  \end{smallmatrix} \right)
\right\}.
\]
admits a $\Q$-defined faithful representation with image
$\widetilde{\mathbf{T}} \leq \GL_{32}$ such that, for every extension
$k$ of $\Q$, the group $\widetilde{\mathbf{T}}(k)$ corresponds to the
subgroup of $\Aut_k(k \otimes_\Z \mathcal{F})$ obtained from extending
the natural action of $\mathbf{T}(k)$ on $\Span_k \langle X, Y, Z
\rangle$ to an action on $k \otimes_\Z \mathcal{F}$:
\[
  \begin{pmatrix}
    \lambda & 0 & 0 \\
    0 & \mu & 0 \\
    0 & 0 & \nu
  \end{pmatrix}
  \text{ acts via }
  \begin{cases}
    X &\!\!\mapsto  \lambda X  \\
    Y &\!\!\mapsto  \mu Y\\ 
    Z &\!\!\mapsto  \nu Z  
  \end{cases},
  \qquad
  \begin{pmatrix}
    1&0&0\\
    0&0&1\\
    0&1&0
  \end{pmatrix}
  \text{ acts via }
  \begin{cases}
    X &\!\! \mapsto  X  \\
    Y &\!\! \mapsto  Z \\ 
    Z &\!\! \mapsto  Y  
  \end{cases}.
\]
We remark that the connected component of $\mathbf{T}$ is isomorphic
to the group $\rho_{2,3}(T_2)$ presented in
\cite[Proposition~6.1]{Be11} which defines an integral not satisfying
a functional equation. Thus it is reasonable to investigate
$\mathbf{Aut}(\Lambda)$ in the hope that it will provide similar
behaviour.

We also consider the $\Q$-defined algebraic group $\mathbf{M}$ such
that, for every extension $k$ of $\Q$, the group $\mathbf{M}(k)$
consists of all $k$-automorphisms of $k \otimes_\Z \mathcal{F}$ of the
form
\begin{equation}\label{equ:elmt-M}
  \begin{array}{llll}
    X & \mapsto & X + U, & U \in\gamma_2(k \otimes_\Z \mathcal{F}) \\
    Y & \mapsto & Y + \upsilon XY + \sigma YZ + V, & V \in\gamma_3(k
    \otimes_\Z \mathcal{F})\\ 
    Z & \mapsto & Z + \upsilon XZ + \tau YZ + W, & W \in\gamma_3(k
    \otimes_\Z \mathcal{F}), \\ 
  \end{array}
\end{equation}
where $\upsilon, \sigma, \tau \in k$. Note that
$\widetilde{\mathbf{T}}(k)$ acts faithfully on $(k \otimes_\Z
\mathcal{F})/\gamma_2(k \otimes_\Z \mathcal{F})$, whereas
$\mathbf{M}(k)$ acts trivially on $(k \otimes_\Z
\mathcal{F})/\gamma_2(k \otimes_\Z \mathcal{F})$.  We define $
\mathbf{S} = \mathbf{M} \rtimes \widetilde{\mathbf{T}}$ so that $
\mathbf{S}(k) = \mathbf{M}(k) \rtimes \widetilde{\mathbf{T}}(k) \leq
\Aut_k(k \otimes_\Z \mathcal{F})$.

\begin{lemma}
  The ideal $k \otimes_\Z \mathcal{I}$ is invariant under the action
  of $ \mathbf{S}(k)$.
\end{lemma}

\begin{proof}
  Elements of $\widetilde{\mathbf{T}}(k)$ map $R_1$ and $R_2$ to
  scalar multiples of themselves. Furthermore, elements of
  $\mathbf{M}(k)$ fix $R_1$ and $R_2$ modulo $\Span_k \langle R_1 X,
  R_1 Y, R_1 Z, R_2 X, R_2 Y, R_2 Z \rangle$.  Indeed, for any element
  $\alpha \in \mathbf{M}(k)$ as in~\eqref{equ:elmt-M}, using the fact
  that commutators of length greater than $4$ are trivial, short
  calculations based on the identities~\eqref{equ:base-change}
  and~\eqref{equ:YZYX-identity} yield
  \begin{align*}
    (R_1)\alpha & = R_1+\upsilon YZYX +\sigma ZYZY+\tau YZYY \in k
    \otimes_\Z \mathcal{I}, \\
    (R_2)\alpha & = R_2+\upsilon ZYZX -\tau YZYZ-\sigma ZYZZ \in k
    \otimes_\Z \mathcal{I}. \qedhere
  \end{align*}
\end{proof}

The action of $ \mathbf{S}(k)$ on $k \otimes_\Z \Lambda$ can be
described in terms of a $\Q$-defined morphism $\rho \colon  \mathbf{S}
\rightarrow \mathbf{Aut}(\Lambda)$ to the algebraic automorphism
group.  Let $\mathbf{K} = \ker(\rho)$.

\begin{theorem} \label{thm:Aut(L)} The morphism $\rho \colon
  \mathbf{S} \rightarrow \mathbf{Aut}(\Lambda)$ induces an isomorphism
  \[
  \mathbf{Aut}(\Lambda) = \img(\rho) \cong  \mathbf{S}/\mathbf{K}.
  \]
\end{theorem}

\begin{proof}
  Let $k$ be an extension field of $\Q$ and let $\alpha \in \Aut_k(k
  \otimes_\Z \Lambda)$.  We are to show that $\alpha \in
  ( \mathbf{S}(k))\rho$.  Clearly, $\alpha$ induces an automorphism
  $\overline{\alpha}$ on $(k \otimes_\Z \Lambda)/\gamma_2(k \otimes_\Z
  \Lambda) = \Span_k \langle \overline{x}, \overline{y}, \overline{z}
  \rangle \cong k^3$.  We write $A = (a_{ij}) \in\GL_{3}(k)$ for the
  matrix representing $\overline{\alpha}$ with respect to the
  $k$-basis $\overline{x}, \overline{y}, \overline{z}$.
  
  From $R_2 \equiv -ZYZ$ modulo $\gamma_4(\mathcal{F})$ we deduce
  $(zyz)\alpha \equiv 0$ modulo~$\gamma_4(k \otimes_\Z \Lambda)$.  A
  short computation shows that, modulo $\gamma_4(k \otimes_\Z
  \Lambda)$,
  \begin{align*}
    (zyz) \alpha & \equiv
    (a_{31}x+a_{32}y+a_{33}z)(a_{21}x+a_{22}y+a_{23}z)(a_{31}x
    + a_{32}y+a_{33}z) \\
    & = c_1 xyy + c_2 xzz + c_3 xyx + c_4 xzx + c_5 xyz + c_6 xzy,
  \end{align*}
  where
  \begin{equation} \label{equ:coef}
    \begin{split}
      c_1 = -a_{32}
      \begin{vmatrix}
        a_{21}&a_{22}\\
        a_{31}&a_{32}
      \end{vmatrix}, \quad & c_2 = -a_{33}
      \begin{vmatrix}
        a_{21}&a_{23}\\
        a_{31}&a_{33}
      \end{vmatrix}, \\
      \quad c_3 = -a_{31}
      \begin{vmatrix}
        a_{21}&a_{22}\\
        a_{31}&a_{32}
      \end{vmatrix}, \quad & c_4 = -a_{31}
      \begin{vmatrix}
        a_{21}&a_{23}\\
        a_{31}&a_{33}
      \end{vmatrix},
    \end{split}
  \end{equation}
  and $c_5, c_6$ take more complicated values which do not feature in
  our argument.  Indeed, $yzy \equiv zyz \equiv 0$ modulo $\gamma_4(k
  \otimes_\Z \Lambda)$ and $yzx = xzy - xyz$.  Therefore the
  contributions to $c_1, c_2, c_3, c_4$ are the `obvious ones', as
  described above.

  Since $xyy$, $xzz$, $xyx$, $xzx$, $xyz$, $xzy$ form a $k$-basis for
  $\gamma_3(k \otimes_\Z \Lambda)$ modulo $\gamma_4(k \otimes_\Z
  \Lambda)$, we conclude that the four coefficients listed in
  \eqref{equ:coef} vanish.  On the other hand, since $A$ is
  invertible,
  \[
  \mathrm{rk}
  \begin{pmatrix}
    a_{21}&a_{22}&a_{23}\\
    a_{31}&a_{32}&a_{33}
  \end{pmatrix}
  = 2.
  \]
  We claim that $a_{21} = a_{31} = 0$.  Otherwise $\begin{vmatrix}
    a_{21}&a_{22}\\a_{31}&a_{32} \end{vmatrix}\neq 0$ or
  $\begin{vmatrix} a_{21}&a_{23}\\a_{31}&a_{33} \end{vmatrix} \neq 0$.
  Looking at the coefficients $c_3$, $c_4$ of $xyx$, $xzx$, we
  conclude that $a_{31} = 0$.  Looking at the coefficients $c_1$,
  $c_2$ of $xyy$, $xzz$, this implies that $a_{21} a_{32}^2 = a_{21}
  a_{33}^2 = 0$.  Since $a_{32}\neq 0$ or $a_{33}\neq 0$, this shows
  that $a_{21}=0$.  

  Multiplying $\alpha$ by a suitable element of
  $(\widetilde{\mathbf{T}}(k)) \rho$, we may assume without loss of
  generality that $a_{11}=1$ so that $A$ is of the form
  \[
  A =
  \begin{pmatrix}
    1&a_{12}&a_{13}\\ 0&a_{22}&a_{23}\\
    0&a_{32}&a_{33}
  \end{pmatrix}.
  \]

  Next we derive consequences from the relation $(zxxx) \alpha =
  (zyz)\alpha$, coming from the relation~$R_2$.  Since $\gamma_5(k
  \otimes_\Z \Lambda) = 0$, we have
  \begin{equation} \label{equ:zxxx} (zxxx)\alpha =
    (a_{32}y+a_{33}z) (x+a_{12}y+a_{13}z) (x+a_{12}y+a_{13}z) (x +
    a_{12}y+a_{13}z),
  \end{equation}
  and similarly
  \begin{equation} \label{equ:zyz}
    \begin{split}
      (zyz)\alpha & = (a_{32}y + a_{33}z) (a_{22}y + a_{23}z)(a_{32}y
      + a_{33}z) \\
      & \qquad + (b_zxy+c_zxz+d_zyz)(a_{22}y+a_{23}z)(a_{32}y+a_{33}z)\\
      & \qquad + (a_{32}y+a_{33}z)(b_yxy+c_yxz+d_yyz)(a_{32}y+a_{33}z)\\
      & \qquad +
      (a_{32}y+a_{33}z)(a_{22}y+a_{23}z)(b_zxy+c_zxz+d_zyz),
    \end{split}
  \end{equation}
  where we are writing
  \begin{align*}
    y\alpha &\equiv a_{22}y + a_{23}z + b_yxy + c_yxz + d_yyz\\
    z\alpha & \equiv a_{32}y + a_{33}z + b_zxy + c_zxz+d_zyz
  \end{align*}
  modulo $\gamma_3(k \otimes_\Z \Lambda)$. 

  Using $yzy = -zyy = yxxx$ and $zyz = -yzz = zxxx$, we can express
  the right-hand side of \eqref{equ:zyz} as a linear combination of
  commutators of length~$4$.  This enables us to apply
  Lemma~\ref{lem:same-weight}.  Comparing, first of all, the
  coefficients of $yxxx = yzy = -zyy$ in~\eqref{equ:zxxx} and
  \eqref{equ:zyz}, we get
  \begin{equation} \label{equ:a32} a_{32} \left( 1 +
      \begin{vmatrix}
        a_{22}&a_{23}\\
        a_{32}&a_{33}
      \end{vmatrix}
    \right) = 0.
  \end{equation}
  Similarly, comparing the coefficients of $zxxx = zyz = -yzz$ in
  \eqref{equ:zxxx} and \eqref{equ:zyz}, we obtain
  \begin{equation} \label{equ:a33} a_{33} \left( 1 -
      \begin{vmatrix}
        a_{22}&a_{23}\\
        a_{32}&a_{33}
      \end{vmatrix}
    \right) =0.
  \end{equation}
  The equations \eqref{equ:a32} and \eqref{equ:a33} yield
  $2a_{32}a_{33}=0$, hence $a_{32}=0$ or $a_{33}=0$.  Because $A$ is
  invertible, $a_{32}=0$ would imply $a_{22} \neq 0$ and $a_{33} \neq
  0$.  Likewise $a_{33}=0$ would imply $a_{23} \neq 0$ and $a_{32}
  \neq 0$.  Left-multiplication by $\left( \begin{smallmatrix}
      1&0&0\\
      0&0&1\\
      0&1&0
    \end{smallmatrix} \right)$ exchanges the second and the third row
  of $A$.  Thus, multiplying $\alpha$ by a suitable element of
  $(\widetilde{\mathbf{T}}(k))\rho$, we may assume without loss of
  generality that $a_{32} = 0$ and $a_{22}, a_{33} \neq 0$.  The
  equation \eqref{equ:a33} now implies that $a_{22}a_{33}=1$, and so
  multiplying $\alpha$ by a suitable element of
  $(\widetilde{\mathbf{T}}(k))\rho$, we may assume that $a_{22} =
  a_{33} = 1$ so that
  \[
  A =
  \begin{pmatrix}
    1&a_{12}&a_{13}\\
    0&1&a_{23}\\
    0&0&1
  \end{pmatrix}.
  \]

  The equations \eqref{equ:zxxx} and \eqref{equ:zyz} simplify to
  \begin{equation} \label{equ:zxxx-simpler} (zxxx)\alpha = z
    (x+a_{12}y+a_{13}z) (x+a_{12}y+a_{13}z) (x + a_{12}y+a_{13}z)
  \end{equation}
  and
  \begin{equation} \label{equ:zyz-simpler}
    \begin{split}
      (zyz)\alpha = z y z &
      + (b_zxy+c_zxz+d_zyz)(y + a_{23}z) z \\
      & + z (b_yxy + c_yxz + d_yyz) z \\
      & + z y (b_zxy+c_zxz).
    \end{split}
  \end{equation}
  Still using $(zxxx)\alpha = (zyz)\alpha$ and applying
  Lemma~\ref{lem:same-weight}, we consider the coefficients for $zxzx
  = zxxz=-xzxz=-xzzx$ in \eqref{equ:zxxx-simpler} and
  \eqref{equ:zyz-simpler}.  Non-zero contributions only come
  from~\eqref{equ:zxxx-simpler}, and we get $2a_{13}=0$, hence
  $a_{13}=0$.

  Next we derive consequences from the relation $(yxxx) \alpha =
  (yzy)\alpha$, coming from the relation~$R_2$.  Since $\gamma_5(k
  \otimes_\Z \Lambda) = 0$, we have
  \begin{equation} \label{equ:yxxx} (yxxx)\alpha = (y+a_{23}z)
    (x+a_{12}y) (x+a_{12}y) (x+a_{12}y),
  \end{equation}
  and similarly
  \begin{equation} \label{equ:yzy}
    \begin{split}
      (yzy)\alpha & = (y + a_{23}z)z(y + a_{23}z) \\
      & \qquad + (b_y xy+c_y xz+d_y yz)z(y+a_{23}z)\\
      & \qquad + (y+a_{23}z)(b_z xy+c_z xz+d_z yz)(y+a_{23}z)\\
      & \qquad + (y+a_{23}z)z(b_y xy+c_y xz+d_y yz),
    \end{split}
  \end{equation}
  where we are writing
  \begin{align*}
    y\alpha & \equiv y + a_{23}z + b_yxy + c_yxz + d_yyz\\
    z\alpha & \equiv z + b_zxy + c_zxz+d_zyz
  \end{align*}
  modulo $\gamma_3(k \otimes_\Z \Lambda)$, as before.

  The right-hand side of~\eqref{equ:yzy} can be expressed as a linear
  combination of commutators of length~$4$, so we may apply
  Lemma~\ref{lem:same-weight}.  Considering the coefficients for $zxxx
  = -yzz$ in \eqref{equ:yxxx} and \eqref{equ:yzy}, we get
  $a_{23}=-a_{23}$ and so $a_{23}=0$.  Next we consider the
  coefficients for $xyxy=xyyx=-yxyx=-yxxy$ in \eqref{equ:yxxx} and
  \eqref{equ:yzy}.  Non-zero contributions only come from
  \eqref{equ:yxxx}, and we get $a_{12}(yxxy+yxyx)=0$, so
  $-2a_{12}xyxy=0$ and hence $a_{12}=0$.  This means that $A$ is the
  identity matrix so that $\alpha$ acts trivially on $(k \otimes_\Z
  \Lambda)/\gamma_2(k \otimes_\Z \Lambda)$.

  It remains to show that $\alpha \in (\mathbf{M}(k))\rho$.
  Shifting $\alpha$ by a suitable element of $(\mathbf{M}(k))\rho$,
  we find $b_1, b_2, b_3 \in k$ such that, modulo $\gamma_3(k
  \otimes_\Z \Lambda)$,
  \begin{equation} \label{equ:congr-for-xyz}
    \begin{split}
      x \alpha & \equiv x \\
      y \alpha & \equiv y + b_1  xy + b_2  xz \\
      z \alpha & \equiv z + b_3 xy.
    \end{split}
  \end{equation}
  We claim that, in fact, $b_1=b_2=b_3=0$.  Indeed, the relation
  $R_1$ gives
  \begin{equation*}
    0 = (yxxx-yzy)\alpha \\ = -b_1 xyzy - b_2 xzzy - b_1 (yz)(xy) -
    b_2 (yz)(xz) - b_3 y(xy)y.
  \end{equation*}
  From the standard relations $(yz)(xy) = xyzy-xyyz$ and $(yz)(xz) =
  xzzy-xzyz$, compare~\eqref{equ:base-change}, we deduce that
  \begin{equation*} 
    0 = b_1 (2xyzy-xyyz) + b_2 (2xzzy-xzyz) - b_3xyyy.
  \end{equation*}
  Using \eqref{equ:xyzy-identity}, we obtain
  \begin{align*}  
    0 & = b_1 xzyy + b_2 \left(2(2xzyz-xyzz) 
      -xzyz \right) - b_3xyyy\\
    & = b_1 xzyy + b_2 (3xzyz-2xyzz) - b_3 xyyy.
  \end{align*}
  Since all four Lie products involved belong to the
  basis~\eqref{equ:L_basis}, we obtain that $b_1=b_2=b_3=0$. It
  follows that $\alpha$ indeed lies in $(\mathbf{M}(k))\rho$.
\end{proof}


\section{Computation of the zeta function} \label{sec:zeta-comp}

Based on Theorem~\ref{thm:Aut(L)} and using the discussion leading up
to Theorem~\ref{thm:dS-Lu}, we proceed to compute the local
pro-isomorphic zeta functions of the Lie lattice~$\Lambda$ defined in
Section~\ref{sec:descr.of.L}.  By Theorem~\ref{thm:Aut(L)}, the
connected component of the identity of the algebraic automorphism
group $\mathbf{G} = \textbf{Aut}(\Lambda)$ decomposes as
\[
\textbf{G}^\circ = \mathbf{N} \rtimes \mathbf{H},
\]
where $\mathbf{N} = \mathbf{M} \rho$ is the unipotent radical and $
\mathbf{H} = \widetilde{\mathbf{T}}^\circ \rho$ is a diagonal group.
We consider $\textbf{G}$ as a subgroup of $\GL_{25}$, via the
$\Z$-basis $(x,y, \ldots,xzyz)$ of~$\Lambda$ in~\eqref{equ:L_basis}.

We employ the set-up of~\cite[Section~2]{dSLu96} as summarised in
Section~\ref{sec:Malcev}.  Fix a prime $p$ and observe that the
condition $\mathbf{G}(\Q_p) = \mathbf{G}(\Z_p) \mathbf{G}^\circ(\Q_p)$
is satisfied.  We write
\[
G = \mathbf{G}^\circ(\Q_p), \qquad N = \mathbf{N}(\Q_p), \qquad H =
 \mathbf{H}(\Q_p).
\]
The natural $G$-module $V = \Span_{\Q_p} \langle x, y, \ldots, xzyz
\rangle \cong \Q_p^{25}$ contains $L = \Z_p \otimes_\Z \Lambda$ as a
sublattice: $L$ consists of the integral elements of~$V$.  Moreover,
$V$ decomposes into a direct sum $V = U_1 \oplus U_2 \oplus U_3 \oplus
U_4$ of $H$-stable subspaces
\begin{align*}
  U_1 & = \langle x,\, y,\, z \rangle && \text{of dimension
    $r_1 = 3$,} \\
  U_2 & = \langle xy,\, xz,\, yz \rangle && \text{of
    dimension $r_2 = 3$,} \\
  U_3 & = \langle xyy,\, xzz,\, xyz,\, xzy,\, xyx,\,xzx\rangle
  && \text{of dimension $r_3 = 6$,} \\
  U_4 & = \langle xyyy,\, xzzz,\, \ldots,\, xzyz \rangle && \text{of
    dimension $r_4 = 13$.}
\end{align*}
Indeed, the elements of $H$ are precisely the diagonal matrices of the form
\begin{multline} \label{equ:diag} \mathrm{diag}(a, b, c, \,\,\, ab,
  ac, bc, \,\,\, ab^2, ac^2, abc, abc, a^2b, a^2c, \\ ab^3, ac^3,
  a^3b, a^3c, a^2b^2, a^2c^2, a^2bc, a^2bc, a^2bc, abc^2, ab^2c,
  ab^2c, abc^2)
\end{multline}
where $a, b, c \in \Q_p^\times$ such that $a^3 = bc$.
Furthermore, for every element $n \in N \subseteq \GL_{25}(\Q_p)$, the
$3 \times 25$ matrix comprising the first three rows of $n$ has the
form
\begin{equation} \label{equ:first-3-rows}
  \begin{pmatrix}
    1 & 0 & 0 & \quad \alpha_1 & \alpha_2 & \alpha_3 & \quad
    \delta_{1,7} & \delta_{1,8} & \cdots & \delta_{1,25} \\
    0 & 1 & 0 & \quad \upsilon & 0 & \sigma & \quad \delta_{2,7} &
    \delta_{2,8} & \cdots
    & \delta_{2,25} \\
    0 & 0 & 1 & \quad 0 & \upsilon & \tau & \quad \delta_{3,7} &
    \delta_{3,8} & \cdots & \delta_{3,25},
  \end{pmatrix},
\end{equation}
with $\alpha_1, \alpha_2, \alpha_3, \upsilon, \sigma, \tau,
\delta_{i,j} \in \Q_p$ for $(i,j) \in \{1,2,3\} \times
\{7,8,\ldots,25\}$.  Observe that the entries in other rows are
uniquely determined by those in the first three rows, due to our
choice of basis.

For $i \in \{1,\ldots,5\}$ we put $V_i = \gamma_i(\Q_p \otimes_\Z
\Lambda)$, so $V_i=U_i\oplus \cdots\oplus U_4$; in particular $V_5
=0$.  For $i \in \{2,\ldots,5\}$ let $N_{i-1}$ be the kernel of
the action of $N$ on~$V/V_i$, so in particular $N_1=N$ and $N_4=1$.
Recall from Section~\ref{sec:Malcev} that the natural representation
of $G/N_{i-1}$ in $\Aut(V/V_i)$ is used to define the integral
elements $(G/N_{i-1})^+$ of $G/N_{i-1}$. In order to apply
Theorem~\ref{thm:dS-Lu}, we need to confirm that the `lifting
condition' \cite[Assumption~2.3]{dSLu96} is satisfied.  As stated in
Section~\ref{sec:Malcev}, this condition asserts that for every $g_0
N_{i-1} \in (G/N_{i-1})^+$ there exists $g \in G^+$ such that $g_0
N_{i-1} = g N_{i-1}$.  Let $g_0 = n_0 h$, with $n_0 \in N$ and
$h\in H$, such that $g_0 N_{i-1} \in (G/N_{i-1})^+$.  Referring to the
description of general elements of $N$ in~\eqref{equ:first-3-rows},
choose $n \in N$ to have the same entries as $n_0$ in the first three
rows, except that those parameters $\delta_{1,j}, \delta_{2,j},
\delta_{3,j}$ of $n$ effecting contributions from~$V_i$ are taken to
be~$0$.  Putting $g = nh$, we conclude that $g_0 N_{i-1} = g N_{i-1}$
and that all entries of $g$ in the first three rows are integral,
hence $(x)g, (y)g, (z)g \in L$.  By Lemma~\ref{lem:enough-xyz}, $g \in
G^+$, hence the lifting condition is indeed satisfied.

Using Theorem~\ref{thm:dS-Lu}, it is now straightforward to calculate
the integral we are interested in.  We first calculate
$\theta_{i-1}(h)$ for $i \in \{2,3,4\}$ and $h \in H$, as defined in
Section~\ref{sec:Malcev}.  Recall that the free parameters in a matrix
$n\in N$ all come from the top three rows.  Moreover,
Lemma~\ref{lem:enough-xyz} shows that to test whether $(n
N_i)\tau_i(h) \in \Mat_{s_i,r_i}(\Z_p)$ for $n N_i \in N_{i-1}/N_i$,
it is necessary and sufficient to check the corresponding entries in
the first three rows; see Section~\ref{sec:Malcev} for an explanation
of how $N_{i-1}/N_i$ is identified with a $\Q_p$-vector space.

For $h = \mathrm{diag} (a, b, c, \ldots ) \in H$, we read off
from~\eqref{equ:diag} that
\begin{equation*}
  \theta_1(h) = \lvert a^3 b^4 c^4 \rvert_p^{-1}\min\{\lvert
  b\rvert_p^{-1}, \lvert c \rvert_p^{-1}\}, \quad 
  \theta_2(h) = \lvert a^{24} b^{15} c^{15} \rvert_p^{-1}, \quad
  \theta_3(h) = \lvert a^{66} b^{45} c^{45} \rvert_p^{-1}
\end{equation*}
and $\det(h) = a^{33} b^{23} c^{23}$.  We define the auxiliary
set
\[
\mathcal{X} = \left\{ (a,b) \in \Z_p^2 \mid a,b \ne 0 \text{
  and } \lvert b \rvert_p \geq \lvert a \rvert_p^3 \right\} \subseteq
(\Q_p^\times)^2.
\]
Using Theorem~\ref{thm:dS-Lu}, this yields
\begin{align*}
  \int_{G^+} & \lvert \det g \rvert_p^s \, d\mu_G(g) = \int_{H^+}
  \lvert \det h \rvert_p^s \prod_{i=1}^3 \theta_i(h) \, d\mu_H(h) \\
  & = \int_{h = \mathrm{diag}(a,b,c,\ldots) \in H^+} \lvert a
  \rvert_p^{-93+33s} \lvert b \rvert_p^{-64+23s} \lvert c
  \rvert_p^{-64+23s} \min\{\lvert b\rvert_p^{-1}, \lvert c
  \rvert_p^{-1}\}
  \, d\mu_H(h) \\
  & = \int_{\substack{(a,b) \in \mathcal{X}, \\ c = a^3b^{-1}}} \,
  \lvert bc \rvert_p^{-31+11s} \lvert b \rvert_p^{-64+23s} \lvert c
  \rvert_p^{-64+23s} \min\{\lvert b\rvert_p^{-1}, \lvert c
  \rvert_p^{-1}\} \, d\mu_{\Q_p^\times}(a) \, d\mu_{\Q_p^\times}(b) \\
  & = \int_{\substack{(a,b) \in \mathcal{X}, \\ c=a^3b^{-1}}} \,
  \lvert bc \rvert_p^{-95+34s} \min\{\lvert b\rvert_p^{-1}, \lvert c
  \rvert_p^{-1}\} \, d\mu_{\Q_p^\times}(a) \, d\mu_{\Q_p^\times}(b) \\
  & = \sum_{\substack{i,j\geq 0 \\ 3 \mid (i+j)}} p^{(i+j)(95-34s)}
  \min\{p^i, p^j\} \, \mu_{\Q_p^\times}
  \!\!\left(p^{(i+j)/3}\Z_p^\times \right) \,
  \mu_{\Q_p^\times} \!\!\left( p^i\Z_p^\times \right)  \\
  & = \sum_{\substack{i,j \geq 0 \\ 3 \mid (i+j)}} \min\{p^i, p^j\}
  X^{i+j} \,\Big\vert_{X=p^{95-34s}} \, ,
\end{align*}
where $\mu_{\Q_p^\times}$ denotes the Haar measure on the
multiplicative group $\Q_p^\times$, normalised so that
$\mu_{\Q_p^\times}(\Z_p^\times) = 1$.  Now
\begin{align*}
  \sum_{\substack{i,j\geq 0\\  3 \mid (i+j)}} & \min\{p^i,
  p^j\} X^{i+j} = \sum_{m,n\geq 0} \min\{p^{3m}, p^{3n}\} X^{3m+3n} \\
  & \quad + \sum_{m,n\geq 0} \min\{p^{3m+1}, p^{3n+2}\} X^{3m+3n+3}
  + \sum_{m,n\geq 0} \min\{p^{3m+2}, p^{3n+1}\} X^{3m+3n+3}\\
  & = \sum_{m,n\geq 0} \min\{p^{3m}, p^{3n}\} X^{3m+3n} + 2
  pX^3 \sum_{m,n\geq 0} \min\{p^{3m}, p^{3n+1}\} X^{3m+3n}.
\end{align*}
We calculate these two pieces separately:
\begin{align*}
  \sum_{m,n\geq 0} \min\{p^{3m}, p^{3n}\} X^{3m+3n} & = 2\sum_{m,k
    \geq 0} p^{3m} X^{6m + 3k} - \sum_{m \geq
    0} p^{3m} X^{6m} \\
  & = \frac{1+X^3}{(1-X^3)(1-p^3X^6)}
\end{align*}
and 
\begin{align*}
  2 p X^3 \sum_{m,n\geq 0} \min\{p^{3m}, p^{3n+1}\} X^{3m+3n} 
  & = 2p X^3 \left( \sum_{n \geq m \geq 0} p^{3m} X^{3m + 3n} +
    \sum_{m > n \geq 0} p^{3n+1} X^{3m+3n}
  \right) \\
  & = 2 p X^3 \left( \sum_{m,k \geq 0} p^{3m} X^{6m+3k} +
    \sum_{n \geq 0}  \sum_{k > 0} p^{3n+1} X^{6n+3k} \right) \\
  & = \frac{2pX^3+2p^2X^6}{(1-X^3)(1-p^3X^6)}.
\end{align*}
Summing the two expressions, we arrive at
\begin{align*}
  \int_{G^+} \lvert \det g \rvert_p^s \, d\mu_G(g) & =
  \frac{1+X^3+2pX^3+2p^2X^6}{(1-X^3)(1-p^3X^6)}
  \,\Big\vert_{X=p^{95-34s}} \\
  & = \frac{1 + p^{285-102s} + 2p^{286-102s} + 2p^{572-204s}}{(1 -
    p^{285-102s}) (1 - p^{573-204s})}.
\end{align*}
Applying Propositions~\ref{pro:Malcev-corr} and
\ref{pro:integral-formula}, we obtain Theorem~\ref{thm:main-thm}.


\appendix

\section{A basis for the free nilpotent Lie algebra of class $4$ on
  $3$ generators} \label{sec:appendix_A}

Let $\mathcal{F}$ be the free nilpotent $\Z$-Lie ring of class $c$ on
$n$ generators $X_1, \ldots, X_n$.  We use simple product notation for
the Lie bracket and left-normed notation.  The Hall collection process
yields an ordered $\Z$\nobreakdash-basis for $\mathcal{F}$;
cf.~\cite[Ch.~4]{Re93}.  The elements $X_1, \ldots, X_n$ are basic
elements of weight $1$, and we order them as $X_1 < \ldots < X_n$.
Basic elements of higher weight $w \geq 2$ are defined inductively as
follows.  If $C_1$ and $C_2$ are basic elements of weights $w_1$ and
$w_2$ such that $w = w_1 + w_2$, then $B = C_1 C_2$ is a basic element
of weight $w$ provided that (i) $C_1 > C_2$ and (ii) if $C_1 = D_1
D_2$ for basic elements $D_1, D_2$, then $D_2 \leq C_2$.  If $B$ is a
basic element of weight $w$, then $C < B$ for any basic element $C$ of
weight less than $w$.  Moreover, if $C_1, C_2, C_3, C_4$ are basic
elements such that $B_1 = C_1 C_2$ and $B_2 = C_3 C_4$ are basic
elements of weight $w$, then $B_1 < B_2$ if one of the following
holds: (i) $C_1 < C_3$, (ii) $C_1 = C_3$ and $C_2 < C_4$.  It is well
known that the basic elements of weight up to $c$ provide a $\Z$-basis
for $\mathcal{F}$.  In fact, for each $i \in \{1, \ldots, c\}$, the
basic elements of weight $i$ induce a $\Z$-basis for the abelian Lie
lattice $\gamma_i(\mathcal{F})/\gamma_{i+1}(\mathcal{F})$.

We are interested in the case $c = 4$ and $n = 3$.  Writing $X =
X_1$, $Y = X_2$ and $Z = X_3$, we obtain $32$ basic elements of
weight up to $4$.  They are, in the described order,
\begin{equation} \label{equ:hall-basis}
  \begin{split}
    & X,\, Y,\, Z, \quad
    YX,\, ZX,\, ZY, \\
    & YXX,\, YXY,\, YXZ,\, ZXX,\, ZXY,\, ZXZ,\,
    ZYY,\, ZYZ, \\
    & (ZX)(YX),\, (ZY)(YX),\, (ZY)(ZX), \\
    & YXXX,\, YXXY,\, \underline{YXXZ},\, YXYY,\, \underline{YXYZ},\, YXZZ, \\
    & ZXXX,\, ZXXY,\, ZXXZ,\, ZXYY,\, \underline{ZXYZ},\, ZXZZ, \\
    & ZYYY,\, ZYYZ,\, ZYZZ.
  \end{split}
\end{equation}
For our computations it is slightly easier to work with left-normed
products as basis elements.  Using the relations
\begin{equation} \label{equ:base-change}
  \begin{split}
    (ZX)(YX) & = \underline{YXXZ} - YXZX \\
    (ZY)(YX) & = XYZY - \underline{XYYZ} \\
    (ZY)(ZX) & = XZZY - \underline{XZYZ},
  \end{split}
\end{equation}
where we have underlined terms already occurring up to a sign change
in~\eqref{equ:hall-basis}, one sees that it is permissible to replace
the three basis elements which are not left-normed, i.e., $(ZX)(YX)$,
$(ZY)(YX)$ and $(ZY)(ZX)$, by $XYZX$, $XYZY$ and $XZZY$.  Reordering
the resulting basis, we arrive at the basis displayed
in~\eqref{equ:F_basis}.

The relations~\eqref{equ:base-change} are obtained as follows.  For
elements $T,U,V,W$ of any Lie ring, the Jacobi identity -- applied to
$T, U, VW$ -- yields
\begin{equation} \label{equ:TUVW}
  \begin{split}
    (TU)(VW) & = (T(VW))U + T(U(VW)) \\
    & = VWUT - VWTU.
  \end{split}
\end{equation}
Suitable substitutions for $T,U,V,W$ now provide the relations for
$X,Y,Z$.

Applying~\eqref{equ:TUVW} and the Jacobi identity, we derive another
useful identity for elements $P,Q,R,S$ of any Lie ring, namely
\begin{equation} \label{equ:PQRS-appendix}
  \begin{split}
    PQRS & = PQSR + (PQ)(RS) \\
    & = (PQS)R + RSQP - RSPQ   \\
    & = ((PS)Q + P(QS))R + RSQP - RSPQ  \\
    & = PSQR + SQPR + RSQP + SRPQ.
  \end{split}
\end{equation}
This is the identity~\eqref{equ:PQRS} stated earlier in the paper.


\end{document}